\documentclass[12pt,a4]{amsart}
\usepackage{amsmath}
\usepackage{amssymb}
\usepackage{amsthm}
\usepackage{enumitem}
\usepackage{url}
\usepackage{soul}
\usepackage{times}
\usepackage[T1]{fontenc}
\usepackage{color}
\usepackage{dsfont}
\usepackage{epsfig}
\usepackage[hidelinks]{hyperref}
\usepackage{setspace}
\usepackage{epstopdf}
\usepackage[numbers,sort]{natbib}
\usepackage{comment}
\usepackage{booktabs}
\usepackage{float}
\usepackage{algorithm}
\usepackage{stmaryrd}
\usepackage{subcaption,graphicx}
\usepackage{float}
\numberwithin{equation}{section}

\usepackage[font=small,labelfont=bf]{caption}
\DeclareCaptionFont{tiny}{\tiny}
\usepackage{geometry}
\theoremstyle{plain}
\newtheorem{proposition}{Proposition}[section]

\newtheorem{theorem}{Theorem}[section]
\newtheorem{lemma}{Lemma}[section]
\newtheorem{corollary}{Corollary}[section]

\theoremstyle{definition}
\newtheorem{definition}{Definition}[section]
\newtheorem{assumption}{Assumption}[section]

\theoremstyle{remark}
\newtheorem{rk}{Remark}[section]
\expandafter\let\expandafter\oldproof\csname\string\proof\endcsname
\let\oldendproof\endproof
\renewenvironment{proof}[1][\proofname]{%
  \oldproof[\noindent\textbf{#1.} ]%
}{\oldendproof}

\newcommand{\E}{\mathbb{E}}

\newcommand{\be}{\begin{equation}}
\newcommand{\ee}{\end{equation}}
\newcommand{\by}{\begin{eqnarray*}}
\newcommand{\ey}{\end{eqnarray*}}
\def\*#1{\mathbf{#1}}

\renewcommand{\leq}{\leqslant}
\renewcommand{\geq}{\geqslant}
\usepackage{xcolor}
\definecolor{dark-red}{rgb}{0.4,0.15,0.15}
\definecolor{dark-blue}{rgb}{0.15,0.15,0.4}
\definecolor{medium-blue}{rgb}{0,0,0.5}
\hypersetup{
    colorlinks, linkcolor={red},
    citecolor={blue}, urlcolor={blue}
}
\allowdisplaybreaks

\begin{document}
%\setstretch{1.3}
\title[Universality of the Langevin diffusion as scaling limit]{Universality of the Langevin diffusion as scaling limit of a family of Metropolis-Hastings processes I: fixed dimension}
\author{Michael C.H. Choi}
\address{Institute for Data and Decision Analytics, The Chinese University of Hong Kong, Shenzhen, Guangdong, 518172, P.R. China.}
\email{michaelchoi@cuhk.edu.cn}
%\address{School of Operations Research and Information Engineering, Cornell University, Ithaca, New York}
%\email{cc2373@cornell.edu}
\date{\today}
\maketitle

%\begin{abstract}
%	We study two types of Metropolis-Hastings reversiblizations for non-reversible Markov chains. Consider a Markov kernel $P$ with stationary measure $\pi$ and its time-reversal denoted by $P^*$. Inspired by the classical Metropolis transition kernel $M_1$, we introduce a self-adjoint kernel $M_2$ that captures the opposite transition effect of $M_1$. This permits us to write $P+P^* = M_1 + M_2$, which leads to bounds on the spectral gap of $P$ in terms of $M_1$ and $M_2$. We obtain an expansion of $P$ (and $P^*$) in terms of the spectral measures of $M_1$ and $M_2$. In the spirit of \citet{Fill91} and \citet{Paulin15}, we introduce a new pseudo-spectral gap based on $M_1$ and $M_2$, and show that the total variation distance from stationarity can be bounded by this gap. We give variance bounds of the Markov chain in terms of the gap.
%\end{abstract}

\begin{abstract}
	Given a target distribution $\mu$ on a general state space $\mathcal{X}$ and a proposal Markov jump process with generator $Q$, the purpose of this paper is to investigate two universal properties enjoyed by two types of Metropolis-Hastings (MH) processes with generators $M_1(Q,\mu)$ and $M_2(Q,\mu)$ respectively. First, we motivate our study of $M_2$ by offering a geometric interpretation of $M_1$, $M_2$ and their convex combinations as $L^1$ minimizers between $Q$ and the set of $\mu$-reversible generators of Markov jump processes. Second, specializing into the case of $\mathcal{X} = \mathbb{R}^d$ along with a Gaussian proposal with vanishing variance and Gibbs target distribution, we prove that, upon appropriate scaling in time, the family of Markov jump processes corresponding to $M_1$, $M_2$ or their convex combinations all converge weakly to an universal Langevin diffusion. While $M_1$ and $M_2$ are seemingly different stochastic dynamics, it is perhaps surprising that they share these two universal properties. These two results are known for $M_1$ in \citet{BD01} and \citet{GM91}, and the counterpart results for $M_2$ and their convex combinations are new.
	\smallskip
	
	
\end{abstract}

%\tableofcontents

%\newpage

\section{Introduction}\label{sec:intro}

The Metropolis-Hastings (MH) algorithm, the Langevin diffusion and their various variants are among the most popular algorithms in the area of Markov chain Monte Carlo (MCMC), see for instance the survey \citet{RR04} and the references therein. Under a Gaussian proposal with vanishing variance and Gibbs target distribution, \citet{GM91} proves that the MH process converges weakly to the Langevin diffusion, thus highlighting the asymptotic connection between this two classes of Markov processes. Another interesting property enjoyed by the MH algorithm, first shown in \citet{BD01}, is that the MH transition kernel minimizes certain $L^1$ distance between the proposal chain and the set of transition kernels that are reversible with respect to the target distribution, thus offering a geometric perspective towards the study of MH algorithm.

With the above classical results in mind, the aim of this paper is to investigate how these two properties can perhaps be extended to an entirely different dynamics that we call the second MH process, introduced recently by the author in \citet{Choi16, ChoiHwang19}. The first universal property is stated in Theorem \ref{thm:geomM1M2} below: both the classical MH and the second MH minimize certain $L^1$ distance, extending the results by \citet{BD01} to a continuous-time and general state space setting. In our main result Theorem \ref{thm:main} below, we state the second universal property: we show that upon the same scaling in time and in space, perhaps surprisingly both the classical MH and the second MH converge to an universal Langevin diffusion. On a microscopic level, both the classical MH and the second MH exhibit different Markovian dynamics, yet however on a macroscopic level or on a large time-scale, both processes and their convex combinations converge to an universal rescaled Langevin diffusion, thus the dynamics of this family are not that different afterall. As emphasized in the title of this paper, we note that the dimension is kept fixed in our weak convergence result. In a related line of work, known as the optimal scaling of MCMC (see for example \citet{RGG97, RR01,B07, JLM14, MPS12, BR17}), the weak convergence results therein are obtained by taking the dimension going to infinity. In the sequel of this paper \citet{Choi19}, we shall investigate the scaling limit of the second MH process in the Curie-Weiss model in the setting of optimal scaling as the dimension increases, in hope of obtaining interesting counterpart results of \citet{BR17}. 

The rest of this paper is organized as follows. In Section \ref{sec:prelim}, we recall the classical and the second MH process and fix our notations. The geometric interpretation of these processes are proved in Section \ref{subsec:geom}. In Section \ref{sec:main}, the weak convergence result is stated, which will be proved in Section \ref{sec:proof}.

\section{Preliminaries}\label{sec:prelim}

\subsection{Metropolis-Hastings generators: $M_1$ and $M_2$}\label{subsec:M1M2}

In this section, we recall the construction of continuous-time Metropolis-Hastings (MH) Markov processes on a general state space $\mathcal{X}$. There are two inputs to the MH algorithm, namely the target distribution and the proposal chain. We refer readers to \citet{RR04} and the references therein for further pointers on this subject. We denote by $\mu$ to be our target distribution and $Q$ to be the generator of the proposal Markov jump process. We assume that both $\mu$ and $Q(x,\cdot)$ are absolutely continuous with respect to a common sigma-finite reference measure $\nu$ on $\mathcal{X}$, and with a slight abuse of notations we still denote their densities by $\mu$ and $Q(x,\cdot)$ respectively. Recall that $Q$ is the generator of a Markov jump process in the sense of \cite[Chapter $4$ Section $2$]{EK86} if and only if 
$$\sup_{x \in \mathcal{X}} \int_{y;~y \neq x} Q(x,y) \nu(dy) < \infty.$$ With these notations, we can now define the first MH generator as a transformation from $Q$ and $\mu$:

\begin{definition}[The first MH generator]\label{def:M1}
Given a target distribution $\mu$ on general state space $\mathcal{X}$ and a proposal continuous-time Markov jump process with generator $Q$, the first MH Markov process is a $\mu$-reversible Markov jump process with generator given by $M_1 = M_1(Q,\mu)$, where for bounded $f$
\begin{align*}
	M_1 f (x) &= \int_{y;~y \neq x} (f(y) - f(x)) M_1(x,y) \nu(dy), \\
	M_1(x,y) &:= \min\left\{Q(x,y),\dfrac{\mu(y)}{\mu(x)}Q(y,x)\right\},\quad x \neq y.
\end{align*}
\end{definition}

Note that 
$$\sup_{x \in \mathcal{X}} \int_{y;~y \neq x} M_1(x,y) \nu(dy) \leq \sup_{x \in \mathcal{X}} \int_{y;~y \neq x} Q(x,y) \nu(dy) < \infty.$$

In view of the earlier work by the author \citet{Choi16, ChoiHwang19}, we would like to study the so-called second MH generator that replaces $\min$ by $\max$ in Definition \ref{def:M1}. More precisely, we define it as follows.

\begin{definition}[The second MH generator]\label{def:M2}
	Given a target distribution $\mu$ on general state space $\mathcal{X}$ and a proposal continuous-time Markov jump process with generator $Q$, define 
	$$M_2(x,y) := \max\left\{Q(x,y),\dfrac{\mu(y)}{\mu(x)}Q(y,x)\right\},\quad x \neq y.$$
	If 
	\begin{align}\label{eq:M2reg}
	\sup_{x \in \mathcal{X}} \int_{y;~y \neq x} M_2(x,y) \nu(dy) < \infty,
	\end{align}
	then the second MH Markov process is a $\mu$-reversible Markov jump process with generator given by $M_2 = M_2(Q,\mu)$, where for bounded $f$
	\begin{align*}
	M_2 f (x) &= \int_{y;~y \neq x} (f(y) - f(x)) M_2(x,y) \nu(dy).
	\end{align*}
\end{definition}

%\begin{definition}[The second MH generator]\label{def:M2}
%	Given a target distribution $\mu$ on finite state space $\mathcal{X}$ and a proposal continuous-time irreducible Markov chain with generator $Q$, the second MH Markov chain has generator given by $M_2 = M_2(Q,\mu) = (M_2(x,y))_{x,y \in \mathcal{X}}$, where
%	$$M_2(x,y) := \begin{cases} \max\left\{Q(x,y),\dfrac{\mu(y)}{\mu(x)}Q(y,x)\right\}, &\mbox{if } x \neq y; \\
%	- \sum_{z:z \neq x} M_2(x,z), & \mbox{if } x = y. \end{cases}$$
%\end{definition}

Comparing Definition \ref{def:M1} and \ref{def:M2}, we see that in the former $M_1$ is always a generator of Markov jump process, while in the latter additional conditions on $\mu$ and $Q$ are required so as to ensure \eqref{eq:M2reg}. In our main results Section \ref{sec:main}, we will consider the special case when $Q(x,\cdot)$ is a normal distribution with mean $x$ and variance $\epsilon$, and $\mu$ is the Gibbs distribution. Under the usual regularity conditions on $\mu$ as in \citet{GM91}, we will see that $M_2$ as defined is a valid generator of a Markov jump process, see Proposition \ref{prop:M1M2} below.

\subsection{Geometric interpretation of $M_1$ and $M_2$}\label{subsec:geom}

In order to motivate the definition of $M_1$ and $M_2$ as natural transformations from $Q$ and $\mu$, in this section we offer a geometric interpretation for both $M_1$ and $M_2$, extending the results by \citet{BD01,ChoiHwang19} to a continuous-time and general state space setting. In our result Theorem \ref{thm:geomM1M2} below, we prove that both $M_1$ and $M_2$, as well as their convex combinations, minimize certain $L^1$ distance between $Q$ and the set of $\mu$-reversible  generator of Markov jump processes on $\mathcal{X}$. As a result, in this sense they are natural transformations that maps a given generator $Q$ of Markov jump process to the set of $\mu$-reversible generators of Markov jump process.

We first introduce a few notations and define a metric to quantify the distance between two generators of Markov jump processes. We write $\mathcal{R}(\mu)$ to be the set of conservative $\mu$-reversible generators of Markov jump processes and $\mathcal{S}(\mathcal{X})$ to be the set of generators of Markov jump processes on $\mathcal{X}$. For any $Q_1, Q_2 \in \mathcal{S}(\mathcal{X})$, similar to \cite[Section $4$]{BD01} we define a metric $d_{\mu}$ on $\mathcal{S}(\mathcal{X})$ to be
$$d_{\mu}(Q_1,Q_2) := \int_{\mathcal{X}\times \mathcal{X} \backslash \Delta} \mu(x) |Q_1(x,y)-Q_2(x,y)| \,\nu(dx) \nu(dy),$$
where $\Delta := \{(x,x);~x \in \mathcal{X}\}$ is the set of diagonal in $\mathcal{X} \times \mathcal{X}$. The distance between $Q$ and $\mathcal{R}(\mu)$ is then defined to be
\begin{align}\label{eq:l1metric}
d_{\mu}(Q,\mathcal{R}(\mu)) := \inf_{R \in \mathcal{R}(\mu)} d_{\mu}(Q,R).
\end{align}
With the above notations in mind, we are now ready to state our  result in this section:

\begin{theorem}\label{thm:geomM1M2}
	Suppose that $Q$ and $\mu$ are such that \eqref{eq:M2reg} is satisfied and $M_2$ is a generator of Markov jump process. The convex combinations $\alpha M_1 + (1-\alpha)M_2$ for $\alpha \in [0,1]$ minimize the distance $d_{\mu}$ between $Q$ and $\mathcal{R}(\mu)$. That is,
	$$d_{\mu}(Q,\mathcal{R}(\mu))= d_{\mu}(Q,\alpha M_1 + (1-\alpha)M_2).$$
	%Moreover, $M_1$ (resp.~ $M_2$) is the unique closest element of $\mathcal{R}(\mu)$ that is coordinate-wise no larger (resp.~no smaller) than $Q$ off-diagonally.
\end{theorem}

%We now proceed to give a proof of Theorem \ref{thm:geomM1M2}.

\begin{proof}
	The proof is inspired by the proof of Theorem $1$ in \citet{BD01} and Theorem $3.1$ in \citet{ChoiHwang19}. We first define two helpful half spaces:
	\begin{align*}
		H^{<} = H^{<}(Q,\mu) &:= \big\{(x,y);~\mu(x)Q(x,y) < \mu(y)Q(y,x)\big\}, \\
		H^{>} = H^{>}(Q,\mu) &:= \big\{(x,y);~\mu(x)Q(x,y) > \mu(y)Q(y,x)\big\}.
	\end{align*}
	We now show that for $R \in \mathcal{R}(\mu)$, $d_{\mu}(Q,R) \geq d_{\mu}(Q,M_2)$. First, we note that
	\begin{align*}
	d_{\mu}(Q,R) \geq \int_{(x,y) \in H^{<}} \big[\mu(x) |Q(x,y) - R(x,y)| + \mu(y) |Q(y,x) - R(y,x)|\big]\,\nu(dx)\nu(dy).
	\end{align*}
	As $R$ is $\mu$-reversible, setting $R(x,y) = Q(x,y) + \epsilon_{xy}$ gives $R(y,x) = \frac{\mu(x)}{\mu(y)}(Q(x,y) + \epsilon_{xy})$. Plugging these expressions back yields
	\begin{align*}
		d_{\mu}(Q,N) &\geq \int_{(x,y) \in H^{<}} \Big[\mu(x) |\epsilon_{xy}| + \mu(y) \left|Q(y,x) - \frac{\mu(x)}{\mu(y)}(Q(x,y) + \epsilon_{xy})\right|\Big] \, \nu(dx)\nu(dy) \\
		&=  \int_{(x,y) \in H^{<}} \big[\mu(x) |\epsilon_{xy}| + \left|\mu(y) Q(y,x) - \mu(x)Q(x,y) - \mu(x) \epsilon_{xy}\right|\big] \, \nu(dx)\nu(dy) \\
		&\geq \int_{(x,y) \in H^{<}} \left|\mu(y) Q(y,x) - \mu(x)Q(x,y) \right| \, \nu(dx)\nu(dy) = d_{\mu}(Q,M_2),
	\end{align*}
	where we use the reverse triangle inequality $|a-b| \geq |a| - |b|$ in the second inequality. Similarly, we can show $d_{\mu}(Q,N) \geq d_{\mu}(Q,M_1)$ via substituting $H^{<}$ by $H^{>}$. To see that $d_{\mu}(Q,M_1) = d_{\mu}(Q,M_2)$, we have
	\begin{align*}
	d_{\mu}(Q,M_2) &= \int_{(x,y) \in H^{<}} \left|\mu(y) Q(y,x) - \mu(x)Q(x,y) \right| \, \nu(dx)\nu(dy)\\
	&= \int_{(y,x) \in H^{>}} \left|\mu(y) Q(y,x) - \mu(x)Q(x,y) \right| \, \nu(dx)\nu(dy) = d_{\mu}(Q,M_1).
	\end{align*}
	As for convex combinations of $M_1$ and $M_2$, we see that
	\begin{align*}
	d_{\mu}(Q,\alpha M_1 + (1-\alpha)M_2) &= (1-\alpha)\int_{(x,y) \in H^{<}} \left|\mu(y) Q(y,x) - \mu(x)Q(x,y) \right| \, \nu(dx)\nu(dy) \\
	&\quad + \alpha \int_{(x,y) \in H^{>}} \left|\mu(y) Q(y,x) - \mu(x)Q(x,y) \right| \, \nu(dx)\nu(dy)\\
	&= (1-\alpha) d_{\mu}(Q,M_2) + \alpha d_{\mu}(Q,M_1) = d_{\mu}(Q,M_1).
	\end{align*}
\end{proof}

\section{Main results: universality of Langevin diffusion as scaling limit of random walk $M_1$ and $M_2$}\label{sec:main}

In this section, we specialize into the case of $\mathcal{X} = \mathbb{R}^{d^*}$ with $d^* \in \mathbb{N}$, and we take the reference measure $\nu$ to be the Lebesgue measure. Let $U: \mathbb{R}^{d^*} \to \mathbb{R}$ be a function satisfying the following regularity assumption:
\begin{assumption}\label{assumpt:U}
	$U$ is continuously differentiable, and its gradient $\nabla U$ is bounded and Lipschitz continuous.
\end{assumption}
Note that the same assumption on $U$ is imposed in \citet{GM91} to obtain their weak convergence result that we will briefly recall later in this section. The target distribution $\mu$ is the Gibbs distribution at temperature $T > 0$ with density given by
\begin{align}\label{eq:Gibbsmu}
\mu(\mathbf{x}) = \dfrac{e^{-U(\mathbf{x})/T}}{\int e^{-U(\mathbf{x})/T} \, d \mathbf{x}}, \quad \mathbf{x} \in \mathbb{R}^{d^*}.
\end{align}
Writing $\phi_{\epsilon}$ to be the density of one-dimensional normal distribution with mean $0$ and variance $\epsilon > 0$, for the proposal Markov jump process, we take $Q_{\epsilon}(\mathbf{x},\mathbf{y})$ to be 
\begin{align}\label{eq:GaussianQ}
Q_{\epsilon}(\mathbf{x},\mathbf{y}) = \dfrac{1}{d^*} \sum_{i=1}^{d^*} \phi_{\epsilon}(y_i - x_i) \prod_{j\neq i} \delta(y_j - x_j ),
\end{align}
where $\delta$ is the Dirac delta function. In words, we pick one of the $d^*$ coordinates uniformly at random, say $i$, and propose a new state at $i$ according to a normal distribution centered at $x_i$ and variance $\epsilon$ while keeping other coordinates unchanged. Note that $Q_{\epsilon}(\mathbf{x},\mathbf{y}) = Q_{\epsilon}(\mathbf{y},\mathbf{x})$. If we write $x_+ := \max\{x,0\}$, we define $s_{M_1}$ and $s_{M_2}$ to be respectively
\begin{align*}
	s_{M_1}(\mathbf{x},\mathbf{y}) &:= e^{-(U(\mathbf{y}) - U(\mathbf{x}))_+/T}, \quad s_{M_1}(i,\mathbf{x},y_i) := s_{M_1}((x_1,\ldots,x_{d^*}),(x_1,\ldots,x_{i-1},y_i,x_{i+1},\ldots,x_{d^*})),\\
	s_{M_2}(\mathbf{x},\mathbf{y}) &:= e^{(U(\mathbf{x}) - U(\mathbf{y}))_+/T}, \quad s_{M_2}(i,\mathbf{x},y_i) := s_{M_2}((x_1,\ldots,x_{d^*}),(x_1,\ldots,x_{i-1},y_i,x_{i+1},\ldots,x_{d^*})).
\end{align*}
With the above notations, we can define $M_1$ and $M_2$ in this setting:

\begin{proposition}[$M_1$ and $M_2$ under Gibbs $\mu$ and Gaussian proposal $Q_{\epsilon}$]\label{prop:M1M2}
	Suppose that $U$ satisfies Assumption \ref{assumpt:U}, $\mu$ is the Gibbs distribution \eqref{eq:Gibbsmu} and $Q_{\epsilon}$ is the Gaussian proposal \eqref{eq:GaussianQ}. Then both $M_1^{\epsilon} = M_1(Q_{\epsilon},\mu)$ and $M_2^{\epsilon} = M_2(Q_{\epsilon},\mu)$ are generators of Markov jump process. Furthermore, for $\mathbf{x} \neq \mathbf{y}$,
	\begin{align*}
		M_1^{\epsilon}(\mathbf{x},\mathbf{y}) &= \dfrac{1}{d^*} \sum_{i=1}^{d^*} s_{M_1}(i,\mathbf{x},y_i) \phi_{\epsilon}(y_i - x_i) \prod_{j\neq i} \delta(y_j - x_j ),\\
		M_2^{\epsilon}(\mathbf{x},\mathbf{y}) &= \dfrac{1}{d^*} \sum_{i=1}^{d^*} s_{M_2}(i,\mathbf{x},y_i) \phi_{\epsilon}(y_i - x_i) \prod_{j\neq i} \delta(y_j - x_j ).
	\end{align*} 
	We write $X^{M_1^{\epsilon}} = (X^{M_1^{\epsilon}}(t))_{t \geq 0}$ and $X^{M_2^{\epsilon}} = (X^{M_2^{\epsilon}}(t))_{t \geq 0}$ to be the Markov jump process with generator $M_1^{\epsilon}$ and $M_2^{\epsilon}$ respectively.
\end{proposition}

\begin{proof}
	We first prove the two formulae for $M_1^{\epsilon}(\mathbf{x},\mathbf{y})$ and $M_2^{\epsilon}(\mathbf{x},\mathbf{y})$. As $Q_{\epsilon}(\mathbf{x},\mathbf{y}) = Q_{\epsilon}(\mathbf{y},\mathbf{x})$, we have 
	\begin{align*}
		M_1^{\epsilon}(\mathbf{x},\mathbf{y}) &= Q_{\epsilon}(\mathbf{x},\mathbf{y}) \min\bigg\{\dfrac{\mu(\mathbf{y})}{\mu(\mathbf{x})},1\bigg\} = Q_{\epsilon}(\mathbf{x},\mathbf{y}) s_{M_1}(\mathbf{x},\mathbf{y}) = \dfrac{1}{d^*} \sum_{i=1}^{d^*} s_{M_1}(i,\mathbf{x},y_i) \phi_{\epsilon}(y_i - x_i) \prod_{j\neq i} \delta(y_j - x_j ), \\
		M_2^{\epsilon}(\mathbf{x},\mathbf{y}) &= Q_{\epsilon}(\mathbf{x},\mathbf{y}) \max\bigg\{\dfrac{\mu(\mathbf{y})}{\mu(\mathbf{x})},1\bigg\} = Q_{\epsilon}(\mathbf{x},\mathbf{y}) s_{M_2}(\mathbf{x},\mathbf{y}) = \dfrac{1}{d^*} \sum_{i=1}^{d^*} s_{M_2}(i,\mathbf{x},y_i) \phi_{\epsilon}(y_i - x_i) \prod_{j\neq i} \delta(y_j - x_j ).
	\end{align*}
	Next, we show that $M_2^{\epsilon}$ is a valid generator of Markov jump process. Let $M$ be an upper bound on 
	$$\sup_{\mathbf{x} \in \mathbb{R}^{d^*}} |\partial_{x_i} U (x)| \leq M. $$ for all $1 \leq i \leq d^*$. By mean value theorem on $U$, we have
	$$s_{M_2}(i,\mathbf{x},y_i) \leq e^{M|y_i - x_i|/T}.$$
	By writing $Z$ to be a normal random variable with mean $0$ and variance $\epsilon$, this leads to 
	$$\int_{\mathbf{y};~\mathbf{x} \neq \mathbf{y}}M_2^{\epsilon}(\mathbf{x},\mathbf{y})\, d\mathbf{y} \leq \mathbb{E}(e^{M|Z|/T}) < \infty,$$
	where $\mathbb{E}(e^{M|Z|/T})$ is the moment generating function of the half-normal distribution $|Z|$, which is independent of $\mathbf{x}$.
\end{proof}

In our main result of this paper, we are primarily interested in the scaling limit of $X^{M_1^{\epsilon}}$ and $X^{M_2^{\epsilon}}$ upon scaling in time as $\epsilon \to 0$. The scaling in space is embedded in the proposal $Q_{\epsilon}$. Let $D^{d^*}[0,\infty)$ be the space of $\mathbb{R}^{d^*}$-valued functions on $[0,\infty)$ that are right continuous with left limit, equipped with the Skorohod topology. We denote the weak convergence of processes in the Skorohod topology by $\Rightarrow$.

\begin{theorem}[Universality of the Langevin diffusion as scaling limit of $M_1^{\epsilon}$ and $M_2^{\epsilon}$]\label{thm:main}
	Suppose that $U$ satisfies Assumption \ref{assumpt:U}, and we let $X^{M_1^{\epsilon}} = (X^{M_1^{\epsilon}}(t))_{t \geq 0}$ and $X^{M_2^{\epsilon}} = (X^{M_2^{\epsilon}}(t))_{t \geq 0}$ to be the Markov jump process with generator $M_1^{\epsilon}$ and $M_2^{\epsilon}$ respectively, both with initial distribution $X^{M_1^{\epsilon}}(t) = X^{M_2^{\epsilon}}(t) = X_0$ independent of $\epsilon$. Let $X = (X(t))_{t \geq 0}$ be the following time-rescaled Langevin diffusion with $X(0) = X_0$ and stochastic differential equation given by
	\begin{align}\label{eq:langevinsde}
		dX(t) = - \dfrac{\nabla U(X(t))}{2Td^*} \, dt + \dfrac{1}{\sqrt{d^*}} \,d W(t),
	\end{align}
	where $(W(t))_{t \geq 0}$ is the standard $d^*$-dimensional Brownian motion. Then we have
	$$X^{M_1^{\epsilon}}\left(\dfrac{\cdot}{\epsilon}\right) \Rightarrow X(\cdot), \quad X^{M_2^{\epsilon}}\left(\dfrac{\cdot}{\epsilon}\right) \Rightarrow X(\cdot)$$
	weakly in $D^{d^*}[0,\infty)$ as $\epsilon \to 0$.
\end{theorem}

\begin{rk}\label{rk:GM91}
	As noted in the abstract and in Section \ref{sec:intro}, the weak convergence of $X^{M_1^{\epsilon}}$ to the Langevin diffusion is first proved by \citet{GM91}. We shall only prove the case of $X^{M_2^{\epsilon}}$ in Section \ref{subsec:proofmain}.
\end{rk}

\begin{rk}
	Denote by $(\widetilde{X}(t))_{t \geq 0}$ the Langevin diffusion with initial condition $\widetilde{X}(0) = X_0$ and stochastic differential equation 
	$$d\widetilde{X}(t) = - \nabla U(\widetilde{X}(t)) \, dt + \sqrt{2T} \,d W(t).$$
	Denote the clock process by $\tau(t) := t/(2Td^*)$, then $\widetilde{X}(\tau(\cdot))$ has the same law as $X(\cdot)$ satisfying \eqref{eq:langevinsde}. In other words, \eqref{eq:langevinsde} is the Langevin diffusion running at time scale $\tau(\cdot)$.
\end{rk}

It is perhaps surprising that both $M_1^{\epsilon}$ and $M_2^{\epsilon}$ share the same scaling limit, given that they have entirely different dynamics. In fact, any Markov jump process whose generator is a convex combination of $M_1^{\epsilon}$ and $M_2^{\epsilon}$ converges to the same Langevin diffusion:

\begin{corollary}\label{cor:main}
	Suppose that $U$ satisfies Assumption \ref{assumpt:U}, and we let $Y^{\epsilon} = (Y^{\epsilon}(t))_{t \geq 0}$ be the Markov jump process with generator $\alpha M_1^{\epsilon} + (1-\alpha)M_2^{\epsilon}$, $\alpha \in [0,1]$ and initial distribution $Y^{\epsilon}(t) = X_0$ independent of $\epsilon$. Let $X = (X(t))_{t \geq 0}$ be the following time-rescaled Langevin diffusion with $X(0) = X_0$ and stochastic differential equation given by
	$$dX(t) = - \dfrac{\nabla U(X(t))}{2Td^*} \, dt + \dfrac{1}{\sqrt{d^*}} \,d W(t),$$
	where $(W(t))_{t \geq 0}$ is the standard $d^*$-dimensional Brownian motion. Then we have
	$$Y^{\epsilon}\left(\dfrac{\cdot}{\epsilon}\right) \Rightarrow X(\cdot)$$
	weakly in $D^{d^*}[0,\infty)$ as $\epsilon \to 0$.
\end{corollary}

In view of Theorem \ref{thm:geomM1M2} and Corollary \ref{cor:main}, we see that on one hand the convex combination of $\alpha M_1^{\epsilon} + (1-\alpha)M_2^{\epsilon}$ may have different dynamics for different $\alpha$, yet interestingly they all minimize the distance $d_{\mu}$ and converge weakly to the same Langevin diffusion.

The rest of the paper is devoted to the proof of Theorem \ref{thm:main} and Corollary \ref{cor:main} in Section \ref{subsec:proofmain} and \ref{subsec:proofcor} respectively.

\section{Proofs of main results}\label{sec:proof}

\subsection{Proof of Theorem \ref{thm:main}}\label{subsec:proofmain}

For notational convenience, we replace $U(\cdot)$ by $U(\cdot)/T$. In view of Remark \ref{rk:GM91}, we only prove the weak convergence of $X^{M_2^{\epsilon}}$. We let $G$ to be the generator of $X$ described by \eqref{eq:langevinsde}, where for $f \in C^2(\mathbb{R}^{d^*})$,
\begin{align*}
	Gf(\*x) = \dfrac{1}{d^*} \sum_{i=1}^{d^*} \partial_{x_i} U(\*x)\partial_{x_i} f(\mathbf{x}) + \dfrac{1}{2d^*} \sum_{i=1}^{d^*} \partial^2_{x_i}f(\*x).
\end{align*}
Note that since the drift $\nabla U$ is Lipschitz continuous, by \cite[Chapter $8$ Theorem $2.5$]{EK86}, the space of infinitely differentiable functions with compact support $C^{\infty}_c(\mathbb{R}^{d^*})$ forms a core of $G$. Thus to prove the desired weak convergence, by \cite[Chapter $1$ Theorem $6.1$]{EK86} it suffices to prove the uniform convergence of the generator in $C^{\infty}_c(\mathbb{R}^{d^*})$, that is, for $f \in C^{\infty}_c(\mathbb{R}^{d^*})$ as $\epsilon \to 0$ we would like to prove that
\begin{align}\label{eq:uniformconv}
	\sup_{\*x \in \mathbb{R}^{d^*}} |(1/\epsilon)M_2^{\epsilon}f(\*x) - Gf(\*x)| \to 0.
\end{align}

Define for $\*x,\*y \in \mathbb{R}^{d^*}$, $\langle \*x,\*y \rangle := \sum_{i=1}^{d^*} x_iy_i, \quad
\| \*x \|^2 := \langle \*x,\*x \rangle,$
\begin{align*}
	\hat{s}_{M_2}(\*x,\*y) &:= e^{\langle \nabla U(\*x),\*x-\*y\rangle_+}, \quad 
	\hat{s}_{M_2}(i,\mathbf{x},y_i) := \hat{s}_{M_2}((x_1,\ldots,x_{d^*}),(x_1,\ldots,x_{i-1},y_i,x_{i+1},\ldots,x_{d^*})), \\
	g(\*x,\*y) &:= U(\*x) - U(\*y) - \langle \nabla U(\*x),\*x-\*y\rangle.
\end{align*}
We now present three lemmas that will aid our proof, and their proofs are deferred to Section \ref{subsubsec:lem1}, \ref{subsubsec:lem2} and \ref{subsubsec:lem3} respectively. The first auxiliary lemma bounds the distance between $s_{M_2}$ and $\hat{s}_{M_2}$:

\begin{lemma}\label{lem:sshat}
	There exists positive constants $M$ and $c_1$ that only depend on $U$ and $T$ such that
	$$|s_{M_2}(\*x,\*y) - \hat{s}_{M_2}(\*x,\*y)| \leq c_1 e^{M \sum_{i=1}^{d^*}|y_i - x_i|} \|\*y - \*x\|^2.$$
	Consequently, we have
	$$|s_{M_2}(i,\mathbf{x},y_i) - \hat{s}_{M_2}(i,\mathbf{x},y_i)| \leq c_1 e^{M|y_i - x_i|}|y_i-x_i|^2.$$
\end{lemma}

Our next lemma controls the upper bound on Lemma \ref{lem:sshat}.
\begin{lemma}\label{lem:foldednormal}
	Recall that $Z$ follows a normal distribution with mean $0$, variance $\epsilon$ and probability density function $\phi_{\epsilon}$. Then for $t \in \mathbb{R}$ and $\epsilon \to 0$ we have
	\begin{align*}
		\E(e^{t|Z|}|Z|^3) &= \mathcal{O}(\epsilon^{3/2}), \\
		\E(e^{t|Z|}|Z|^4) &= \mathcal{O}(\epsilon^{2}).
	\end{align*}
	
\end{lemma}

With Lemma \ref{lem:sshat} and \ref{lem:foldednormal}, we prove the following estimates on the drift and volatility terms of $M_2^{\epsilon}$ as $\epsilon \to 0$:
\begin{lemma}\label{lem:estimatedrift}
	For $1 \leq i \leq d^*$, as $\epsilon \to 0$, 
	\begin{align}
		(1/\epsilon) \int (y_i - x_i) M_2^{\epsilon}(\*x,\*y)\, d\*y &= -\partial_{x_i}U(\*x)/(2d^*) + \mathcal{O}(\epsilon^{1/2}), \label{eq:drift}\\
		(1/\epsilon) \int (y_i - x_i)^2 M_2^{\epsilon}(\*x,\*y)\, d\*y &= 1/d^* + \mathcal{O}(\epsilon^{1/2}), \label{eq:volatility}\\
		(1/\epsilon) \int (y_i - x_i)^3 M_2^{\epsilon}(\*x,\*y)\, d\*y &= \mathcal{O}(\epsilon^{1/2}), \label{eq:third}
	\end{align}
	where the convergence are all uniform in $\mathbf{x} \in \mathbb{R}^{d^*}$.
\end{lemma}

We proceed to complete the proof of Theorem \ref{thm:main}. By Taylor expansion on $f \in C^{\infty}_c(\mathbb{R}^{d^*})$, there exists $\mathbf{z} = \alpha \mathbf{x} + (1-\alpha)\mathbf{y} $ such that
\begin{align*}
	(1/\epsilon)M_2^{\epsilon}f(\*x) &= (1/\epsilon) \int (f(\*y)- f(\*x)) M_2^{\epsilon}(\*x,\*y)\, d\*y \\
	&= (1/\epsilon) \int \left(\sum_{i=1}^{d^*} \partial_{x_i} f(\mathbf{x})(y_i-x_i) + \dfrac{1}{2} \sum_{i,j=1}^{d^*} \partial_{x_i} \partial_{x_j} f(\mathbf{x})(y_i-x_i)(y_j-x_j)\right)  M_2^{\epsilon}(\*x,\*y)\, d\*y \\
	&\quad + (1/\epsilon) \int \left(\dfrac{1}{6} \sum_{i,j,k=1}^{d^*} \partial_{x_i} \partial_{x_j} \partial_{x_k} f(\mathbf{z})(y_i-x_i)(y_j-x_j)(y_k-x_k)\right)  M_2^{\epsilon}(\*x,\*y)\, d\*y \\
	&= (1/\epsilon) \int \left(\sum_{i=1}^{d^*} \partial_{x_i} f(\mathbf{x})(y_i-x_i) + \dfrac{1}{2} \sum_{i=1}^{d^*} \partial^2_{x_i} f(\mathbf{x})(y_i-x_i)^2\right)  M_2^{\epsilon}(\*x,\*y)\, d\*y \\
	&\quad + (1/\epsilon) \int \left(\dfrac{1}{6} \sum_{i=1}^{d^*} \partial^3_{x_i} f(\mathbf{z})(y_i-x_i)^3\right)  M_2^{\epsilon}(\*x,\*y)\, d\*y \\
	&= \dfrac{1}{d^*} \sum_{i=1}^{d^*} \partial_{x_i} U(\*x)\partial_{x_i} f(\mathbf{x}) + \dfrac{1}{2d^*} \sum_{i=1}^{d^*} \partial^2_{x_i}f(\*x) + \mathcal{O}(\epsilon^{1/2}) \\
	&= Gf(\mathbf{x}) + \mathcal{O}(\epsilon^{1/2}),
\end{align*}
where the fourth equality follows from Lemma \ref{lem:estimatedrift} and the fact that $f$ has compact support. Note that the convergence is uniform in $\mathbf{x}$.

\subsubsection{Proof of Lemma \ref{lem:sshat}}\label{subsubsec:lem1}

	First, by the Taylor expansion on $U$ and the fact that $U$ is Lipschitz continuous by Assumption \ref{assumpt:U}, there exists constant $c_1$ such that
	$$|g(\*x,\*y)| \leq c_1 \| \*y - \*x \|^2.$$
	We would like to show the following inequality holds, by considering the possible signs of $U(\*x) - U(\*y)$ and $\langle \nabla U(\*x),\*x - \*y\rangle$:
	\begin{align}\label{eq:estimate1}
	1 - e^{\langle \nabla U(\*x),\*x-\*y\rangle_+ - (U(\*x) - U(\*y))_+} \leq 1 - e^{-|g(\*x,\*y)|} \leq |g(\*x,\*y)| \leq c_1 \| \*y - \*x \|^2.
	\end{align}
	\begin{itemize}
		\item Case 1: $U(\*x) - U(\*y) > 0$, $\langle \nabla U(\*x),\*x-\*y\rangle > 0$
		\newline
		In this case, since $- \left(U(\*y) - U(\*x) + \langle \nabla U(\*x),\*x-\*y\rangle\right) \leq |g(\*x,\*y)|$, upon rearranging we obtain the leftmost inequality of \eqref{eq:estimate1}.
		
		\item Case 2: $U(\*x) - U(\*y) \leq 0$, $\langle \nabla U(\*x),\*x-\*y\rangle \leq 0$
		\newline
		The leftmost inequality of \eqref{eq:estimate1} holds trivially.
		
		\item Case 3: $U(\*x) - U(\*y) > 0$, $\langle \nabla U(\*x),\*x-\*y\rangle \leq 0$
		%\newline
		\begin{align*}
		1 - e^{\langle \nabla U(\*x),\*x-\*y\rangle_+ - (U(\*x) - U(\*y))_+} = 1 - e^{U(\*y) - U(\*x)} \leq 1 - e^{U(\*y) - U(\*x) + \langle \nabla U(\*x),\*x - \*y \rangle} = 1 - e^{-|g(\*x,\*y)|}. 
		\end{align*}
		
		\item Case 4: $U(\*x) - U(\*y) \leq 0$, $\langle \nabla U(\*x),\*x-\*y\rangle > 0$
		%\newline
		\begin{align*}
		1 - e^{\langle \nabla U(\*x),\*x-\*y\rangle_+ - (U(\*x) - U(\*y))_+} = 1 - e^{\langle \nabla U(\*x),\*x-\*y\rangle} \leq 0 \leq 1 - e^{-|g(\*x,\*y)|}.
		\end{align*}
	\end{itemize}
	Similarly, we would like to show the following inequality holds, by considering the possible signs of $U(\*x) - U(\*y)$ and $\langle \nabla U(\*x),\*x - \*y\rangle$:
	\begin{align}\label{eq:estimate2}
	1 - e^{- \langle \nabla U(\*x),\*x-\*y\rangle_+ + (U(\*x) - U(\*y))_+} \leq 1 - e^{-|g(\*x,\*y)|} \leq |g(\*x,\*y)| \leq c_1 \| \*y - \*x \|^2.
	\end{align}
	\begin{itemize}
		\item Case 1: $U(\*x) - U(\*y) > 0$, $\langle \nabla U(\*x),\*x-\*y\rangle > 0$
		%\newline
		\begin{align*}
		1 - e^{- \langle \nabla U(\*x),\*x-\*y\rangle_+ + (U(\*x) - U(\*y))_+} = 1 - e^{g(\*x,\*y)} \leq 1 - e^{-|g(\*x,\*y)|}.
		\end{align*}
		
		\item Case 2: $U(\*x) - U(\*y) \leq 0$, $\langle \nabla U(\*x),\*x-\*y\rangle \leq 0$
		\newline
		The leftmost inequality of \eqref{eq:estimate2} holds trivially.
		
		\item Case 3: $U(\*x) - U(\*y) > 0$, $\langle \nabla U(\*x),\*x-\*y\rangle \leq 0$
		%\newline
		\begin{align*}
		1 - e^{- \langle \nabla U(\*x),\*x-\*y\rangle_+ + (U(\*x) - U(\*y))_+} = 1 - e^{U(\*x) - U(\*y)} \leq 1 - e^{U(\*y) - U(\*x) + \langle \nabla U(\*x),\*x - \*y \rangle} = 1 - e^{-|g(\*x,\*y)|}. 
		\end{align*}
		
		\item Case 4: $U(\*x) - U(\*y) \leq 0$, $\langle \nabla U(\*x),\*x-\*y\rangle > 0$
		%\newline
		\begin{align*}
		1 - e^{- \langle \nabla U(\*x),\*x-\*y\rangle_+ + (U(\*x) - U(\*y))_+} = 1 - e^{-\langle \nabla U(\*x),\*x-\*y\rangle} \leq 1 - e^{U(\*x) - U(\*y) - \langle \nabla U(\*x),\*x - \*y \rangle} = 1 - e^{g(\*x,\*y)} = 1 - e^{-|g(\*x,\*y)|}.
		\end{align*}
	\end{itemize}
	As a result, collecting both \eqref{eq:estimate1} and \eqref{eq:estimate2} we have 
	\begin{align*}
	s_{M_2}(\*x,\*y) - \hat{s}_{M_2}(\*x,\*y) &= e^{(U(\*x)-U(\*y))_+} \left(1 - e^{\langle \nabla U(\*x),\*x-\*y\rangle_+ - (U(\*x) - U(\*y))_+}\right) \leq c_1 e^{M \sum_{i=1}^{d^*}|y_i - x_i|} \|\*y - \*x\|^2, \\
	\hat{s}_{M_2}(\*x,\*y) - s_{M_2}(\*x,\*y) &= e^{\langle \nabla U(\*x),\*x-\*y\rangle_+} \left(1 - e^{- \langle \nabla U(\*x),\*x-\*y\rangle_+ + (U(\*x) - U(\*y))_+}\right) \leq c_1 e^{M \sum_{i=1}^{d^*}|y_i - x_i|} \|\*y - \*x\|^2,
	\end{align*}
	where the inequalities in the two equations above follow from mean value theorem and the fact that $\nabla U$ is bounded by Assumption \ref{assumpt:U}.

\subsubsection{Proof of Lemma \ref{lem:foldednormal}}\label{subsubsec:lem2}

	Let $h(t) := \E(e^{t|Z|})$ and we write $\Phi(\cdot)$ to be the cumulative distribution function of standard normal. We also denote $\phi^{(i)}_{\epsilon}$ to be the $i$-th derivative of $\phi_{\epsilon}$. By brute force differentiation and integration we note that
	\begin{align*}
	h(t) &= 2 e^{\frac{\epsilon t^2}{2}}(1-\Phi(-\sqrt{\epsilon}t)), \\
	\partial_t h(t) &= \E(e^{t|Z|}|Z|) = 2 \sqrt{\epsilon} e^{\frac{\epsilon t^2}{2}} \phi_1(-\sqrt{\epsilon}t) + 2 \epsilon t e^{\frac{\epsilon t^2}{2}} (1-\Phi(-\sqrt{\epsilon}t)),	 \\
	\partial^2_t h(t) &= \E(e^{t|Z|}|Z|^2) = -2\epsilon e^{\frac{\epsilon t^2}{2}} \phi^{(1)}_1(-\sqrt{\epsilon}t) + 4 \epsilon^{3/2} t e^{\frac{\epsilon t^2}{2}} \phi_1(-\sqrt{\epsilon}t) + 2\epsilon(1-\Phi(-\sqrt{\epsilon}t))(e^{\frac{\epsilon t^2}{2}} + \epsilon t^2 e^{\frac{\epsilon t^2}{2}}), \\
	\partial^3_t h(t) &= \E(e^{t|Z|}|Z|^3) = 2 \epsilon^{3/2} e^{\frac{\epsilon t^2}{2}} \phi_1^{(2)}(-\sqrt{\epsilon}t) - 6 \epsilon^2 t e^{\frac{\epsilon t^2}{2}} \phi_1^{(1)}(-\sqrt{\epsilon}t)  \\
	&\qquad \qquad \qquad + 6 \epsilon^{3/2} \phi_1(-\sqrt{\epsilon}t)(\epsilon t^2 e^{\frac{\epsilon t^2}{2}}  + e^{\frac{\epsilon t^2}{2}}) + 2\epsilon (1- \Phi(-\sqrt{\epsilon}t))(3 \epsilon t e^{\frac{\epsilon t^2}{2}} + \epsilon^2 t^3 e^{\frac{\epsilon t^2}{2}}), \\
	\partial^4_t h(t) &= \E(e^{t|Z|}|Z|^4) = -2 \epsilon^{3/2} e^{\frac{\epsilon t^2}{2}} \phi_1^{(3)}(-\sqrt{\epsilon}t) + 2 \epsilon^{3/2} e^{\frac{\epsilon t^2}{2}} \epsilon t \phi_1^{(2)}(-\sqrt{\epsilon}t)		\\
	&\qquad \qquad \qquad + 6 \epsilon^{5/2} t e^{\frac{\epsilon t^2}{2}} \phi_1^{(2)}(-\sqrt{\epsilon}t) - 6 \epsilon^{2} \phi_1^{(1)}(-\sqrt{\epsilon}t) (e^{\frac{\epsilon t^2}{2}} + \epsilon t^2 e^{\frac{\epsilon t^2}{2}})   \\
	&\qquad \qquad \qquad - 6 \epsilon^{2} \phi_1^{(1)}(-\sqrt{\epsilon}t)(\epsilon t^2 e^{\frac{\epsilon t^2}{2}} + e^{\frac{\epsilon t^2}{2}}) + 6 \epsilon^{3/2} \phi_1(-\sqrt{\epsilon}t)(\epsilon^2 t^3 e^{\frac{\epsilon t^2}{2}} + 3 \epsilon te^{\frac{\epsilon t^2}{2}}) \\
	&\qquad \qquad \qquad  + 2\epsilon^{3/2}\phi_1(-\sqrt{\epsilon}t) (3 \epsilon t e^{\frac{\epsilon t^2}{2}} + \epsilon^2 t^3 e^{\frac{\epsilon t^2}{2}}) + 2\epsilon (1- \Phi(-\sqrt{\epsilon}t))(6 \epsilon^2 t^2 e^{\frac{\epsilon t^2}{2}} + 3\epsilon e^{\frac{\epsilon t^2}{2}} + \epsilon^3 t^4 e^{\frac{\epsilon t^2}{2}}).
	\end{align*}
	Since $\phi_1(-\sqrt{\epsilon}t) = \mathcal{O}(1)$, $\phi_1^{(1)}(-\sqrt{\epsilon}t) = \mathcal{O}(\epsilon^{1/2})$, $\phi_1^{(2)}(-\sqrt{\epsilon}t) = \mathcal{O}(1)$ and $\phi_1^{(3)}(-\sqrt{\epsilon}t) = \mathcal{O}(\epsilon^{1/2})$ as $\epsilon \to 0$, we see that 
	\begin{align*}
	\partial^3_t h(t) &= \E(e^{t|Z|}|Z|^3) = \mathcal{O}(\epsilon^{3/2}), \\
	\partial^4_t h(t) &= \E(e^{t|Z|}|Z|^4) = \mathcal{O}(\epsilon^{2}).
	\end{align*}

\subsubsection{Proof of Lemma \ref{lem:estimatedrift}}\label{subsubsec:lem3}

	We first prove \eqref{eq:drift}. Note that
	\begin{align*}
	(1/\epsilon) \int (y_i - x_i) M_2^{\epsilon}(\*x,\*y)\, d\*y &= \dfrac{1}{d^* \epsilon} \int (y_i - x_i) s_{M_2}(i, \mathbf{x},y_i) \phi_{\epsilon}(y_i-x_i)\, dy_i \\ 
	&= \dfrac{1}{d^* \epsilon} \int (y_i - x_i) \hat{s}_{M_2}(i, \mathbf{x},y_i) \phi_{\epsilon}(y_i-x_i)\, dy_i \\
	&\quad + \dfrac{1}{d^* \epsilon} \int (y_i - x_i) (s_{M_2}(i, \mathbf{x},y_i) - \hat{s}_{M_2}(i, \mathbf{x},y_i)) \phi_{\epsilon}(y_i-x_i)\, dy_i \\
	&= \dfrac{1}{d^* \epsilon} \int (y_i - x_i) \hat{s}_{M_2}(i, \mathbf{x},y_i) \phi_{\epsilon}(y_i-x_i)\, dy_i + \mathcal{O}(\epsilon^{1/2}) \\
	&= \dfrac{1}{d^* \epsilon^{1/2}} \int_{\partial_{x_i}U(\mathbf{x}) z \leq 0} z e^{-\partial_{x_i}U(\mathbf{x}) z \epsilon^{1/2}} \phi_{1}(z)\, dz \\
	&\quad + \dfrac{1}{d^* \epsilon^{1/2}} \int_{\partial_{x_i}U(\mathbf{x}) z > 0} z  \phi_{1}(z)\, dz + \mathcal{O}(\epsilon^{1/2}),
	\end{align*}
	where the third equality follows from Lemma \ref{lem:sshat} and \ref{lem:foldednormal}. On $\{\mathbf{x};~ \partial_{x_i}U(\mathbf{x}) = 0\}$, clearly \eqref{eq:drift} holds uniformly. On $\{\mathbf{x};~ \partial_{x_i}U(\mathbf{x}) > 0\}$, the above equation becomes
	\begin{align*}
	(1/\epsilon) \int (y_i - x_i) M_2^{\epsilon}(\*x,\*y)\, d\*y &= \dfrac{1}{d^* \epsilon^{1/2}} \int_{z \leq 0} z e^{-\partial_{x_i}U(\mathbf{x}) z \epsilon^{1/2}} \phi_{1}(z)\, dz \\
	&\quad + \dfrac{1}{d^* \epsilon^{1/2}} \int_{z > 0} z  \phi_{1}(z)\, dz + \mathcal{O}(\epsilon^{1/2}) \\
	&= \dfrac{1}{d^* \epsilon^{1/2}} \int_{z \leq 0} z e^{ (\partial_{x_i}U(\mathbf{x}))^2\epsilon/2} \phi_{1}(z+ \partial_{x_i}U(\mathbf{x})\epsilon^{1/2})\, dz \\
	&\quad + \dfrac{1}{d^* \epsilon^{1/2}} \int_{z > 0} z  \phi_{1}(z)\, dz + \mathcal{O}(\epsilon^{1/2}) \\
	&= \dfrac{1}{d^* \epsilon^{1/2}} \int_{y \leq \partial_{x_i}U(\mathbf{x})\epsilon^{1/2}} (y-\partial_{x_i}U(\mathbf{x})\epsilon^{1/2}) e^{ (\partial_{x_i}U(\mathbf{x}))^2\epsilon/2} \phi_{1}(y)\, dy \\
	&\quad + \dfrac{1}{d^* \epsilon^{1/2}} \int_{z > 0} z  \phi_{1}(z)\, dz + \mathcal{O}(\epsilon^{1/2}) \\
	&= \dfrac{1}{d^* \epsilon^{1/2}} \int_0^{\mathcal{O}(\epsilon^{1/2})} y \phi_1(y)\,dy - \dfrac{\partial_{x_i}U(\mathbf{x})}{d^*}\left(1/2 + \int_0^{\mathcal{O}(\epsilon^{1/2})} \phi_1(y)\,dy \right) + \mathcal{O}(\epsilon^{1/2}) \\ 
	&= -\partial_{x_i}U(\*x)/(2d^*) + \mathcal{O}(\epsilon^{1/2}),
	\end{align*}
	where we complete the square to obtain the second equality, and we use the fact that $\nabla U$ is bounded by Assumption \ref{assumpt:U} in the fourth equality. Similarly, we can show that \eqref{eq:drift} holds uniformly on  $\{\mathbf{x};~ \partial_{x_i}U(\mathbf{x}) < 0\}$, and hence for all $x \in \mathbb{R}^{d^*}$.
	
	Next, we prove \eqref{eq:volatility}. Note that
	\begin{align*}
	(1/\epsilon) \int (y_i - x_i)^2 M_2^{\epsilon}(\*x,\*y)\, d\*y &= \dfrac{1}{d^* \epsilon} \int (y_i - x_i)^2 s_{M_2}(i, \mathbf{x},y_i) \phi_{\epsilon}(y_i-x_i)\, dy_i \\ 
	&= \dfrac{1}{d^* \epsilon} \int (y_i - x_i)^2 \hat{s}_{M_2}(i, \mathbf{x},y_i) \phi_{\epsilon}(y_i-x_i)\, dy_i \\
	&\quad + \dfrac{1}{d^* \epsilon} \int (y_i - x_i)^2 (s_{M_2}(i, \mathbf{x},y_i) - \hat{s}_{M_2}(i, \mathbf{x},y_i)) \phi_{\epsilon}(y_i-x_i)\, dy_i \\
	&= \dfrac{1}{d^* \epsilon} \int (y_i - x_i)^2 \hat{s}_{M_2}(i, \mathbf{x},y_i) \phi_{\epsilon}(y_i-x_i)\, dy_i + \mathcal{O}(\epsilon) \\
	&= \dfrac{1}{d^*} \int_{\partial_{x_i}U(\mathbf{x}) z \leq 0} z^2 e^{-\partial_{x_i}U(\mathbf{x}) z \epsilon^{1/2}} \phi_{1}(z)\, dz \\
	&\quad + \dfrac{1}{d^* } \int_{\partial_{x_i}U(\mathbf{x}) z > 0} z^2  \phi_{1}(z)\, dz + \mathcal{O}(\epsilon),
	\end{align*}
	where the third equality follows from Lemma \ref{lem:sshat} and \ref{lem:foldednormal}. On $\{\mathbf{x};~ \partial_{x_i}U(\mathbf{x}) = 0\}$, clearly \eqref{eq:drift} holds uniformly. On $\{\mathbf{x};~ \partial_{x_i}U(\mathbf{x}) > 0\}$, the above equation becomes
	\begin{align*}
	(1/\epsilon) \int (y_i - x_i)^2 M_2^{\epsilon}(\*x,\*y)\, d\*y &= \dfrac{1}{d^*} \int_{z \leq 0} z^2 e^{-\partial_{x_i}U(\mathbf{x}) z \epsilon^{1/2}} \phi_{1}(z)\, dz \\
	&\quad + \dfrac{1}{d^*} \int_{z > 0} z^2  \phi_{1}(z)\, dz + \mathcal{O}(\epsilon) \\
	&= \dfrac{1}{d^*} \int_{z \leq 0} z^2 e^{ (\partial_{x_i}U(\mathbf{x}))^2\epsilon/2} \phi_{1}(z+ \partial_{x_i}U(\mathbf{x})\epsilon^{1/2})\, dz \\
	&\quad + \dfrac{1}{d^*} \int_{z > 0} z^2  \phi_{1}(z)\, dz + \mathcal{O}(\epsilon) \\
	&= \dfrac{1}{d^*} \int_{y \leq \partial_{x_i}U(\mathbf{x})\epsilon^{1/2}} (y-\partial_{x_i}U(\mathbf{x})\epsilon^{1/2})^2 e^{ (\partial_{x_i}U(\mathbf{x}))^2\epsilon/2} \phi_{1}(y)\, dy \\
	&\quad + \dfrac{1}{d^*} \int_{z > 0} z^2  \phi_{1}(z)\, dz + \mathcal{O}(\epsilon) \\
	&= \dfrac{1}{d^*} \left(1+\int_0^{\mathcal{O}(\epsilon^{1/2})} y^2 \phi_1(y)\,dy\right) + \mathcal{O}(\epsilon^{1/2}) \\ 
	&= 1/d^* + \mathcal{O}(\epsilon^{1/2}),
	\end{align*}
	where we complete the square to obtain the second equality, and we use the fact that $\nabla U$ is bounded by Assumption \ref{assumpt:U} in the fourth equality. Similarly, we can show that \eqref{eq:volatility} holds uniformly on  $\{\mathbf{x};~ \partial_{x_i}U(\mathbf{x}) < 0\}$, and hence for all $x \in \mathbb{R}^{d^*}$.
	
	Finally, we prove \eqref{eq:third}. Note that
	\begin{align*}
	(1/\epsilon) \int (y_i - x_i)^3 M_2^{\epsilon}(\*x,\*y)\, d\*y &= \dfrac{1}{d^* \epsilon} \int (y_i - x_i)^3 s_{M_2}(i, \mathbf{x},y_i) \phi_{\epsilon}(y_i-x_i)\, dy_i \\ 
	&\leq \dfrac{1}{d^* \epsilon} \int |y_i - x_i|^3 e^{M|y_i-x_i|} \phi_{\epsilon}(y_i-x_i)\, dy_i = \mathcal{O}(\epsilon^{1/2}), 
	\end{align*}
	where the inequality and the second equality follow from the mean value theorem on $U$ and Lemma \ref{lem:foldednormal}.

\subsection{Proof of Corollary \ref{cor:main}}\label{subsec:proofcor}

Similar to the proof of Theorem \ref{thm:main}, it suffices for us to prove that
\begin{align*}
\sup_{\*x \in \mathbb{R}^{d^*}} |(1/\epsilon)(\alpha M_1^{\epsilon} + (1-\alpha)M_2^{\epsilon})f(\*x) - Gf(\*x)| &\leq \sup_{\*x \in \mathbb{R}^{d^*}} \alpha|(1/\epsilon) M_1^{\epsilon}f(\*x) - Gf(\*x)| \\
&\quad + \sup_{\*x \in \mathbb{R}^{d^*}} (1-\alpha)|(1/\epsilon)M_2^{\epsilon}f(\*x) - Gf(\*x)| \\
&\to 0.
\end{align*}

\noindent \textbf{Acknowledgements}.
We would like to thank Konstantin Avrachenkov, Jim Dai, Xuefeng Gao, Lu-Jing Huang, Aaron Smith and Jure Vogrinc for constructive discussions related to this work. Michael Choi acknowledges the support from The Chinese University of Hong Kong, Shenzhen grant PF01001143. 

\bibliographystyle{abbrvnat}
\bibliography{thesis}

\end{document}